\documentclass[11pt, twoside]{amsart}

\usepackage{t1enc}
\usepackage{latexsym}
\usepackage{amssymb}
\usepackage{graphicx}
\usepackage{amsmath}
\usepackage{amsthm}
\usepackage{amsfonts}
\usepackage{mathpazo}
\usepackage{mathrsfs}
\usepackage[paper = letterpaper, left = 2.5cm, right = 2.5cm, headsep = 6mm,
footskip = 10mm, top = 25mm, bottom = 25mm, footnotesep=5mm, headheight =
2cm]{geometry}
\usepackage[all]{xy}
\usepackage[british]{babel}
\usepackage{soul}
\usepackage{hyperref}

\newcommand*{\widebar}{\overline}

\newcommand{\cyclic}{\mathop{\kern0.9ex{{+}\kern-2.10ex\raise-0.20
      ex\hbox{\Large\hbox{$\circlearrowright$}}}}\limits}
\newcommand{\acts}{\mbox{ \raisebox{0.26ex}{\tiny{$\bullet$}} }}

\def\N{\ifmmode{\mathbb N}\else{$\mathbb N$}\fi}
\def\Z{\ifmmode{\mathbb Z}\else{$\mathbb Z$}\fi}
\def\Q{\ifmmode{\mathbb Q}\else{$\mathbb Q$}\fi}
\def\R{\ifmmode{\mathbb R}\else{$\mathbb R$}\fi}
\def\C{\ifmmode{\mathbb C}\else{$\mathbb C$}\fi}
\def\K{\ifmmode{\mathbb K}\else{$\mathbb K$}\fi}
\def\P{\ifmmode{\mathbb P}\else{$\mathbb P$}\fi}
\def\g{\ifmmode{\mathfrak g}\else {$\mathfrak g$}\fi}
\def\h{\ifmmode{\mathfrak h}\else {$\mathfrak h$}\fi}
\def\a{\ifmmode{\mathfrak a}\else {$\mathfrak a$}\fi}
\def\k{\ifmmode{\mathfrak k}\else {$\mathfrak k$}\fi}
\def\p{\ifmmode{\mathfrak p}\else {$\mathfrak p$}\fi}
\def\b{\ifmmode{\mathfrak b}\else {$\mathfrak b$}\fi}
\def\n{\ifmmode{\mathfrak n}\else {$\mathfrak n$}\fi}
\def\m{\ifmmode{\mathfrak m}\else {$\mathfrak m$}\fi}
\def\t{\ifmmode{\mathfrak t}\else {$\mathfrak t$}\fi}
\def\O{\ifmmode{\mathscr{O}}\else {$\mathscr{O}$}\fi}
\def\W{\ifmmode{\mathcal{V}}\else {$\mathscr{W}$}\fi}
\def\id{{\rm id}}

\def\hq{/\hspace{-0.14cm}/}
\def\kleinematrix#1,#2,#3,#4,{\begin{pmatrix}#1 & #2 \\ #3 & #4
  \end{pmatrix}}

\newcommand{\mbb}[1]{\mathbb{#1}}
\newcommand{\wt}[1]{\widetilde{#1}}

\newcommand{\ol}[1]{\overline{#1}}

\DeclareMathOperator{\Ad}{Ad}

\DeclareMathOperator{\Lie}{Lie}

\DeclareMathOperator{\dom}{dom}

\DeclareMathOperator{\Type}{Type}

\DeclareMathOperator{\Image}{Im}

\newtheoremstyle{daniel}{3.0mm}{0mm}{\itshape}{}{\bfseries}{.}{1.5mm}{}
\theoremstyle{daniel}
\newtheorem{thm}{Theorem}[section]
\newtheorem{prop}[thm]{Proposition}
\newtheorem{Defi}[thm]{Definition}
\newtheorem{lemma}[thm]{Lemma}
\newtheorem{cor}[thm]{Corollary}
\newtheorem{Exs}[thm]{Examples}
\newtheorem{Ex}[thm]{Example}
\newtheorem{Rems}[thm]{Remarks}
\newtheorem{Rem}[thm]{Remark}

\newtheorem*{thm*}{Theorem}
\newtheorem*{cor*}{Corollary}
\newtheorem*{thm3.6}{Theorem 3.6}
\newtheorem*{thm4.3}{Theorem 4.3}
\newtheorem*{mainthm}{Main Theorem}
\newtheorem*{prop*}{Proposition}
\newtheorem*{Notation}{Notation}

\newtheorem*{subclaim}{Subclaim}

\newenvironment{rem}   {\begin{Rem}\em}{\end{Rem}}
\newenvironment{rems}   {\begin{Rems}\em}{\end{Rems}}
\newenvironment{defi}  {\begin{Defi}\em}{\end{Defi}}
\newenvironment{ex}  {\begin{Ex}\em}{\end{Ex}}

\numberwithin{equation}{section}
 \def\polhk#1{\setbox0=\hbox{#1}{\ooalign{\hidewidth
  \lower1.5ex\hbox{`}\hidewidth\crcr\unhbox0}}}
  \def\polhk#1{\setbox0=\hbox{#1}{\ooalign{\hidewidth
  \lower1.5ex\hbox{`}\hidewidth\crcr\unhbox0}}}
\providecommand{\bysame}{\leavevmode\hbox to3em{\hrulefill}\thinspace}
\providecommand{\MR}{\relax\ifhmode\unskip\space\fi MR }
\providecommand{\href}[2]{#2}

\begin{document}

\title{Invariant meromorphic functions on Stein spaces}

\author{Daniel Greb and Christian Miebach}

\address{Albert-Ludwigs-Universit\"at\\
Mathematisches Institut\\
Abteilung f\"ur Reine Mathematik\\
Eckerstr.~1\\
79104 Freiburg im Breisgau\\
Germany}
\email{daniel.greb@math.uni-freiburg.de}{}
\urladdr{\href{http://home.mathematik.uni-freiburg.de/dgreb}
{http://home.mathematik.uni-freiburg.de/dgreb}}

\address{Laboratoire de Math\'ematiques Pures et Appliqu\'ees\\
Universit\'e du Littoral\\
50, rue F.~Buisson\\
62228 Calais Cedex\\
France}
\email{miebach@lmpa.univ-littoral.fr}{}
\urladdr{\href{http://www-lmpa.univ-littoral.fr/~miebach/}
{http://www-lmpa.univ-littoral.fr/~miebach/}}



\thanks{The first author acknowledges the partial support by DFG-Forschergruppe
790 ''Classification of Algebraic Surfaces and Compact Complex Manifolds''.
Both authors would like to thank Peter Heinzner for fruitful discussions and
gratefully ac\-know\-ledge the hospitality of the Fakult\"at f\"ur
Mathematik at Ruhr-Universit\"at Bochum where parts of the research presented
in this article were carried out.}

\date{14th October 2010}

\begin{abstract}
In this paper we develop fundamental tools and methods to study meromorphic
functions in an equivariant setup. As our main result we construct quotients of
Rosenlicht-type for Stein spaces acted upon holomorphically by
complex-reductive Lie groups and their algebraic subgroups. In particular, we
show that in this setup invariant meromorphic functions separate orbits in
general position. Applications to almost homogeneous spaces and principal orbit
types are given. Furthermore, we use the main result to investigate the
relation between holomorphic and meromorphic invariants for reductive group
actions. As one important step in our proof we obtain a weak equivariant
analogue of Narasimhan's embedding theorem for Stein spaces.

%\setlength{\parindent}{0em}
%\vspace*{0.2cm}
%
%{\sc{R\'esum\'e.}} Dans ce travail nous d\'eveloppons des outils et des
%m\'ethodes fondamentaux afin d'\'etudier les fonctions m\'eromorphes
%invariantes sur les espaces de Stein $X$ munis d'une action holomorphe d'un
%groupe complexe-r\'eductif $G$. Nous construisons des quotients \`a la
%Rosenlicht pour l'action d'un sous-groupe alg\'ebrique de $G$ sur $X$. En
%particulier on montre que dans cette situation les fonctions m\'eromorphes
%invariantes sous ce sous-groupe alg\'ebrique s\'eparent ses orbites en position
%g\'en\'erale. Nous donnons aussi des applications concernants les espaces
%presque homog\`enes et les types d'orbite principaux. De plus, le r\'esultat
%principal est utilis\'e afin de clarifier la relation entre les invariants
%holomorphes voire m\'eromorphes de $G$. Une \'etape importante de notre preuve
%consiste \`a montrer un faible analogue \'equivariant du th\'eor\`eme de
%Narasimhan sur les plongements propres des espaces de Stein.
\end{abstract}

\maketitle
\section{Introduction}

One of the fundamental results relating invariant theory and the geometry of
algebraic group actions is Rosenlicht's Theorem \cite[Thm.~2]{Rosenlicht2}: for
any action of a linear algebraic group on an algebraic variety there exists a
finite set of invariant rational functions that separate orbits in general
position. Moreover, there exists a rational quotient, i.e., a Zariski-open
invariant subset on which the action admits a geome\-tric quotient. It is the
purpose of this paper to study meromorphic functions invariant under
holomorphic group actions and to construct quotients of Rosenlicht-type in the
analytic category.

Examples of non-algebraic holomorphic actions of $\mbb{C}^*$ on projective
surfaces with nowhere Hausdorff orbit space show that even in the compact
analytic case an analogue of Rosenlicht's Theorem does not hold without
further assumptions. If a complex-reductive group acts \emph{meromorphically}
on a compact K\"ahler space (and more generally a compact complex space of
class~$\mathscr{C}$), existence of meromorphic quotients was shown by
Lieberman~\cite{Lieberman} and Fujiki~\cite{FujikiAutomorphisms}.

As a natural starting point in the non-compact case we consider group actions
on spaces with rich function theory such as Stein spaces. Actions of reductive
groups and their subgroups on these spaces are known to possess many features
of algebraic group actions. However, while the holomorphic invariant theory in
this setup is well understood, cf.~\cite{HeinznerGIT}, invariant meromorphic
functions until now have been less studied.

In this paper we develop fundamental tools to study meromorphic functions in an
equivariant setup. We use these tools to prove the following result, which
provides a natural generalisation of Rosenlicht's Theorem to Stein spaces with
actions of complex-reductive groups.

\begin{mainthm}
Let $H<G$ be an algebraic subgroup of a complex-reductive Lie group $G$ and let
$X$ be a Stein $G$-space. Then, there exists an $H$-invariant
Zariski-open dense subset $\Omega$ in $X$ and a holomorphic map
$p\colon\Omega\to Q$ to a Stein space $Q$ such that
\vspace{-0.25cm}
\begin{enumerate}
\item the map $p$ is a geometric quotient for the $H$-action on $\Omega$,
\item the map $p$ is universal with respect to $H$-stable analytic subsets of
$\Omega$,
\item the map $p$ is a submersion and realises $\Omega$ as a topological fibre
bundle over $Q$,
\item the map $p$ extends to a weakly meromorphic map (in the sense of Stoll)
from $X$ to $Q$,
\item for every $H$-invariant meromorphic function $f\in \mathscr{M}_X(X)^H$,
there exists a unique meromorphic function $\bar f \in \mathscr{M}_Q(Q)$
such that $f|_\Omega = \bar f \circ p$, and
\item the $H$-invariant meromorphic functions on $X$ separate the $H$-orbits in
$\Omega$.
\end{enumerate}
\end{mainthm}

The idea of proof is  to first establish a weak equivariant analogue of
Remmert's and Narasimhan's embedding theorem for Stein spaces
\cite{NarasimhanEmbedding}. More precisely, given a $G$-irreducible Stein
$G$-space we prove the existence of a $G$-equivariant holomorphic map into a
finite-dimensional $G$-repre\-sen\-ta\-tion space $V$ that is a proper
embedding when restricted to a big Zariski-open $G$-invariant subset, see
Proposition~\ref{prop:genericembedding}. Since the $G$-action on $V$ is
algebraic, we may then apply Rosenlicht's Theorem to this linear action.
Subsequently, a careful comparison of algebraic and holomorphic geometric
quotients allows us to carry over the existence of a Rosenlicht-type quotient
from $V$ to $X$.

The geometric quotient constructed in this paper provides us with a new and
effective tool to investigate invariant meromorphic functions on Stein spaces.
In the following we shortly describe two typical applications of our main
result.

Given a Stein $G$-space we show that every  invariant meromorphic function is a
quotient of two invariant holomorphic functions precisely if the generic fibre
of the natural invariant-theoretic quotient $\pi\colon X \to X\hq G$ contains a
dense orbit, see Theorem~\ref{thm:aHquniversal}. An important class of examples
for this situation consists of representation spaces of semisimple groups $G$.

An important fundamental result of
Richardson~\cite{RichardsonPrincipalSteinManifolds} states that in every
connected Stein $G$-manifold there exists an open and dense subset on which all
isotropy groups are conjugate in $G$. Under further assumptions on the group
action we use the Main Theorem to sharpen Richardson's result by showing that
there exists a Zariski-open subset on which the conjugacy class of stabiliser
groups is constant, see Proposition~\ref{prop:genericreductivestab}. In
particular, for every effective torus action on a Stein manifold we find a
Zariski-open subset that is a principal fibre bundle over the meromorphic
quotient, cf.\ Remark~\ref{rem:torusaction}.

This paper is organised as follows. In Section~2 we introduce the necessary
background on actions of complex-reductive groups and on
related notions of quotient spaces. Furthermore, we shortly discuss the main
technical tools used in this paper. In Sections~3 and~4 we give applications
of the Main Theorem as well as some examples which illustrate that the result
does not hold for non-algebraic subgroups of $G$. In Section~5 we establish the
Weak Equivariant Embedding Theorem, before we prove the Main Theorem in the
final Sections~6 and~7.

\section{Preliminaries: Definitions and Tools}

In the following, all \emph{complex spaces} are assumed to be reduced and to
have countable topology. If $\mathscr{F}$ is a sheaf on a complex space $X$,
and $U \subset X$ is an open subset, then $\mathscr{F}(U)$ denotes the set of
sections of $\mathscr{F}$ over $U$. By definition, analytic subsets of complex
spaces are closed. Furthermore, an \emph{algebraic group} is by definition
linear algebraic, i.e., a closed algebraic subgroup of some $GL_N(\C)$.

\subsection{Actions of Lie groups}

If $L$ is a real Lie group, then a \emph{complex $L$-space $Z$} is a complex
space with a real-analytic action $\alpha\colon L\times Z\to Z$ such that all
the maps $\alpha_g\colon Z\to Z$, $z\mapsto\alpha(g,z)=:g\acts z$ are
holomorphic. If $L$ is a complex Lie group, a \emph{holomorphic $L$-space $Z$}
is a complex $L$-space such that the action map $\alpha\colon L\times Z\to Z$
is holomorphic. If $X$ is at the same time a Stein space and a holomorphic
$L$-space, we shortly say that $X$ is a \emph{Stein $L$-space}. A complex
$L$-space is called \emph{$L$-irreducible} if $L$ acts transitively on the set
of irreducible components of $X$. Note that in this case $X$ is
automatically pure-dimensional. If the set $X/L$ of $L$-orbits can be endowed
with the structure of a complex space such that the quotient map $p\colon X\to
X/L$ is holomorphic, then $X$ is $L$-irreducible if and only if $X/L$ is
irreducible. In particular, under this condition $\mathscr{M}_{X/L}(X/L)$ and
$\mathscr{M}_X(X)^L$ are fields.

\subsection{Geometric quotients}

One of the main tasks in the proof of the Main Theorem is the construction of a
geometric quotients for the action of complex Lie groups on complex spaces in
the sense of the following definition.

\begin{defi}\label{defi:geometricquotient}
Let $L$ be a complex Lie group and let $X$ be a holomorphic $L$-space. A
\emph{geometric quotient} for the action of $L$ on $X$ is a holomorphic map
$p\colon X\to Q$ onto a complex space $Q$ such that
\begin{enumerate}
\item for all $x\in X$, we have $p^{-1}\bigl(p(x)\bigr)=L\acts x$,
\item $Q$ has the quotient topology with respect to $p$,
\item $(\pi_*\mathscr{O}_X)^L= \mathscr{O}_Q$.
\end{enumerate}
\end{defi}

If a geometric quotient $p\colon X\to Q$ for the action of $L$ exists, we can
identify $Q$ with the set of $L$-orbits in $X$ and we will often write $X/L$
instead of $Q$. The map $p\colon X\to Q$ has the following universality
property: for any $G$-invariant holomorphic map $\phi\colon X \to Y$ into a
complex space $Y$ there exists a uniquely defined holomorphic map $\ol{\phi}: Q
\to Y$ such that $\phi = \ol{\phi} \circ p$. We call a geometric quotient
$p\colon X \to Q$ \emph{universal with respect to invariant analytic subsets}
if for every such set $A \subset X$ the restriction $p|_A\colon A \to p(A)$ is
a geometric quotient for the $L$-action on $A$. Note that $p(A)$ is always an
analytic subset of $Q$, see Lemma~\ref{lem:imageundergeomquotmap}.

\begin{rem}
We also use the corresponding concepts in the algebraic category. Note that in
this case item (2) of Definition~\ref{defi:geometricquotient} requires the
quotient to have the quotient Zariski-topology with respect to the map $p$.
\end{rem}

The following general existence result for geometric quotients in the algebraic
category by Rosenlicht is the starting point of this paper.

\begin{thm}[Thm.~2 of \cite{Rosenlicht2}]\label{thm:Rosenlicht}
Let $H$ be a linear algebraic group and $X$ an $H$-irreducible algebraic
$H$-variety. Then, there exists an $H$-invariant Zariski-open dense subset $U$
of $X$ that admits a geometric quotient. Furthermore, this quotient fulfills
$\C(U/H)=\C(X)^H$.
\end{thm}

\subsection{Analytic Hilbert quotients and slice-type
stratification}\label{subsect:strat}

Let $G$ be a complex-reductive Lie group and $X$ a holomorphic $G$-space. A
complex space $Y$ together with a $G$-invariant surjective holomorphic map
$\pi\colon X \to Y$ is called an \emph{analytic Hilbert quotient} of $X$ by the
action of $G$ if
\begin{enumerate}
 \item $\pi$ is a locally Stein map, and
 \item $(\pi_*\mathscr{O}_X)^G = \mathscr{O}_Y$ holds.
\end{enumerate}
An analytic Hilbert quotient of a holomorphic $G$-space $X$ is unique up to
biholomorphism once it exists, and we will denote it by $X\hq G$. This is the
natural analogue of the concept of \emph{good quotient} or \emph{categorical
quotient} in Algebraic Geometry, cf.\ \cite[Ch.~3]{BBSurvey}. Moreover, if $X$
is an algebraic $G$-variety with a good quotient $\pi\colon X\to X \hq G$, then
the associated map $\pi^h\colon X^h \to (X\hq G)^h$ is an analytic Hilbert
quotient, see \cite{Lunaalgebraicanalytic}. If $X$ is a Stein $G$-space, then
the analytic Hilbert quotient $\pi\colon X \to X\hq G$ exists, see \cite{Snow},
\cite{HeinznerGIT}. It has the following properties, cf.~\cite{HMP}:
\begin{enumerate}
\item Given a $G$-invariant holomorphic map $\phi\colon X \to Z$ to a
complex space $Z$, there exists a unique holomorphic map $\widebar \phi\colon
X\hq G \to Z$ such that $\phi = \widebar \phi \circ \pi$.
\item For every Stein subspace $A$ of $X\hq G$ the inverse image $\pi^{-1}(A)$
is a Stein subspace of $X$.
\item If $A_1$ and $A_2$ are $G$-invariant analytic (in particular, closed)
subsets of $X$, then we have $\pi(A_1) \cap \pi(A_2) = \pi(A_1 \cap A_2)$.
\item For a $G$-invariant closed complex subspace $A$ of $X$, which is defined
by a $G$-invariant sheaf $\mathscr{I}_A$ of ideals, the image sheaf
$(\pi_*\mathscr{I}_A)^G$ endows the image $\pi(A)$ in $X\hq G$ with the
structure of a closed complex subspace of $X\hq G$. Moreover, the restriction
of $\pi$ to $A$ is an analytic Hilbert quotient for the action of $G$ on $A$.
\end{enumerate}

It follows that two points $x,x'\in X$ have the same image in $X\hq G$ if and
only if $\overline{G\acts x} \cap \overline{G\acts x'} \neq \emptyset$. For
each $q \in X\hq G$, the fibre $\pi^{-1}(q)$ contains a unique closed $G$-orbit
$G\acts x$. The orbit $G\acts x$ is affine (see \cite[Prop. 2.3 and 2.5]{Snow})
and hence, the stabiliser $G_x$ of $x$ in $G$ is a complex-reductive Lie group
by a result of Matsushima.

Let $G$ be a complex-reductive Lie group and let $X$ be a holomorphic
$G$-space with analytic Hilbert quotient. There exist two related important
stratifications of the quotient $X\hq G$. The main reference for these
stratifications in the algebraic case is \cite[Sect.~III.2]{LunaSlice}. In the
following  we are going to use the notion of \emph{orbit type} and of
\emph{slice type} as defined in \cite[Sect.~4]{HeinznerEinbettungen}: If $X$ is
a $G$-irreducible Stein $G$-space and $q \in X\hq G$, then there exists a
unique closed $G$-orbit $G\acts x$ in the fibre $\pi^{-1}(q)$. We define the
slice type of $q$ to be the type of the $G_x$-representation on the Zariski
tangent space $T_xX$, i.e., the isomorphism class of the $G$-vector bundle $G
\times_{G_x} (T_xX)$. Analogously, we denote by $\Type(G\acts x)$ the orbit
type of $x$, i.e., the conjugacy class $(G_x)$ in $G$ of the isotropy subgroup
$G_x$ of $G$ at $x$. Using the holomorphic slice theorem and the corresponding
results in the algebraic category one obtains the following result.

\begin{prop}\label{prop:Lunastrat}
Let $X$ be a $G$-irreducible holomorphic $G$-space with analytic Hilbert
quotient $\pi\colon X \to X\hq G$. The decomposition of $X\hq G$ according to
slice type defines a complex analytic stratification of the quotient $X\hq G$.
In particular, in $X\hq G$ there exists a maximal, Zariski-open stratum
$S_\mathrm{max}$ of the slice-type stratification. The orbit-type of closed
orbits is constant on this stratum. Furthermore, the restriction of $\pi$ to
$X_{\mathrm{max}} := \pi^{-1}(S_\mathrm{max})$ realises
$X_{\mathrm{max}}$ as a holomorphic fibre bundle over $S_\mathrm{max}$.
\end{prop}

\subsection{(Weakly) meromorphic maps and functions}

Recall from~\cite{RemmertAbbildungen} and~\cite[sect.~4.6]{Fischer} that there
is a natural correspondence between meromorphic functions on a
(pure-dimensional) complex space $X$ and so-called meromorphic graphs, i.e.,
graphs of meromorphic maps from $X$ to $\P_1=\C\cup \{\infty\}$ that do not map
any irreducible component of $X$ to $\infty$. For a meromorphic function $f$ on
$X$ we denote by $P_f$ the \emph{pole variety} of $f$. It is a nowhere dense
analytic subset of $X$, and the smallest subset of $X$ such that $f$ is
holomorphic on $X \setminus P_f$. We set $\text{dom}(f) := X \setminus
P_f$ and we call it the \emph{domain of definition} of $f$.

Suppose now that a complex Lie group $L$ acts on a complex space $X$. Then we
have an induced action of $L$ on the algebra $\mathscr{M}_X(X)$ of meromorphic
functions as follows. Let $f$ be a meromorphic  function on $X$ with graph
$\Gamma_f\subset X\times\mbb{P}_1$. The group $L$ acts on $X\times\P_1$ by the
$L$-action on the first factor. Given $g \in L$, we define a new meromorphic
graph $\Gamma_{g\acts f}:=g\acts \Gamma_f \subset X \times \P_1$ and hence a
meromorphic function $g\acts f$ on $X$. In this way we obtain an action of $L$
on $\mathscr{M}_X(X)$ by algebra homomorphisms. A meromorphic function $f \in
\mathscr{M}_X(X)$ is $L$-invariant if and only if its graph $\Gamma_f$ is an
$L$-invariant analytic subset of $X \times \P_1$. In this case the pole variety
of $f$ is an $L$-invariant analytic subset of $X$.

The following definition is taken from~\cite{StollMeromorpheAbbildungenI} and
\cite{StollMeromorpheAbbildungenII}. It is useful when considering maps into
(non-compactifiable) non-compact target spaces, as we will do in the following.

\begin{defi}\label{defi:meromorphic}
Let $X$ be a complex space and let $A$ be a nowhere dense analytic subset of
$X$. Let $Y$ be a complex space. Then, a holomorphic map $\phi\colon X\setminus
A \to Y$ is called \emph{weakly meromorphic}, if for any point $p_0 \in A$ and
any one-dimensional complex submanifold $C$ of $X$ with $C \cap A = \ol C
\cap A = \{p_0\}$ there exists at most one point $q_0 \in Y$ with the following
property: there exists a sequence  $(p_n)_{n\in \N} \subset C \setminus A$ with
$\lim_{n \to \infty } p_n = p_0$ such that $q_0$ is the accumulation point of
$\bigl(\phi(p_n)\bigr)_{n \in \N}$.
\end{defi}

\begin{ex}
A meromorphic map is in particular weakly meromorphic, see~\cite[Satz
3.3]{StollMeromorpheAbbildungenI}.
\end{ex}

\section{Applications}

In the following sections we give applications of the Main Theorem to almost-homogeneous spaces (Section~\ref{subsect:almosthomogeneous}), to the problem of
realising meromorphic invariants as quotients of holomorphic invariants
(Section~\ref{subsect:quotientinvariants}), to holomorphically convex spaces
with actions of compact groups (Section~\ref{subsect:holconvex}), to the
existence of principal orbits types (Section~\ref{subsect:Richardson}), and to
actions of unipotent groups (Section~\ref{subsect:unipotent}).

\subsection{Characterising almost-homogeneous
spaces}\label{subsect:almosthomogeneous}

Using Rosenlicht's Theorem, we see that an algebraic variety is
almost-homogeneous for a group $G$ acting on $X$ if and only if every
$G$-invariant rational function on $X$ is constant. As a corollary of our
main result we obtain the corresponding result in the complex-analytic
category:

\begin{prop}\label{prop:Hopenorbit}
Let $X$ be a Stein $G$-space and let $H$ be an algebraic subgroup of $G$. Then
$H$ has an open orbit in $X$ if and only if $\mathscr{M}_X(X)^H = \C$.
\end{prop}

\begin{rem}
Note that this can also be proven without using the Main Theorem: in case
$\mathscr{M}_X(X)^H=\C$, clearly also $\mathscr{O}_X(X)^G=\C$. A result of Snow
\cite[Cor.~5.6]{Snow} implies that $X$ naturally carries the structure of an
affine algebraic $G$-variety. Consequently, $X$ is almost-homogeneous by
Rosenlicht's Theorem, see Theorem~\ref{thm:Rosenlicht}.
\end{rem}

\subsection{The connection between holomorphic and meromorphic invariants}\label{subsect:quotientinvariants}

Let $X$ be a $G$-irre\-du\-ci\-ble Stein $G$-space with analytic Hilbert quotient
$\pi\colon X \to X\hq G$. Then clearly the field of invariant meromorphic
functions $\mathscr{M}_X(X)^G$ contains the field $\mathscr{M}_{X\hq G}(X\hq
G)$ of meromorphic functions on $X \hq G$ via the pull-back morphism
\begin{equation}\label{eq:pullback}
\pi^*\colon  \mathscr{M}_{X\hq G}(X\hq G) \hookrightarrow \mathscr{M}_X(X)^G.
\end{equation}
Using our main result, in this section we describe the image of $\pi^*$ in
$\mathscr{M}_X(X)^G$ and we characterise those spaces for which $\pi^*$ is an
isomorphism. Furthermore, we give examples why in general this cannot be
expected. It follows that in most situations the information encoded in the
Rosenlich-type quotient constructed in the Main Theorem cannot be recovered
from the analytic Hilbert quotient $X\hq G$.

The following result characterises the image of $\pi^*$ in $\mathscr{M}_X(X)^G$.

\begin{prop}
Let $X$ be a $G$-irreducible Stein $G$-space with analytic Hilbert quotient
$\pi\colon X \to X\hq G$. A function $f \in \mathscr{M}_X(X)^G$ is contained in
$\Image(\pi^*)$ if and only if there exist $p,q \in \mathscr{O}_X(X)^G$, $q\neq
0$, such that $f = p/q$.
\end{prop}

\begin{proof}
If $f = p/q$ for $p,q \in\mathscr{O}_X(X)^G$, $q \not \equiv 0$, then by the
universal properties of the analytic Hilbert quotient there exist holomorphic
functions $\bar p$ and $\bar q$ on $X\hq G$ such $p = \pi^*\bar p$ and $q =
\pi^*\bar q$, respectively. Consequently, $f$ is the pull-back of the
meromorphic function $\bar p / \bar q \in \mathscr{M}_{X\hq G}(X\hq G)$.
Conversely, assume that $f= \pi^* \bar f$ for some $\bar f\in \mathscr{O}_{X\hq
G}(X\hq G)$. Since $X\hq G$ is Stein, the Poincar\'e problem on $X\hq G$ is
universally solvable, see for example \cite[Ch.~4, \S2,
Thm.~4]{TheorySteinSpacesAlt}; i.e., there exist holomorphic functions $\bar p,
\bar q \in \mathscr{O}_{X\hq G}(X\hq G)$, $q \not \equiv 0 $, such that $\bar f
= \bar p / \bar q$. Then, $f = \pi^* \bar p / \pi^* \bar q$ is a quotient of
holomorphic invariants.
\end{proof}

The following example shows that the inclusion~\eqref{eq:pullback} is
strict in general.

\begin{ex}\label{ex:CstarOnCtwo}
Consider the action of $\C^*$ on $\C^2$ by scalar multiplication. Then, the
meromorphic function $f(z,w) = z/w$ is $\C^*$-invariant. However, the analytic
Hilbert quotient is $\pi\colon \C^2 \to\{\text{point}\}$, so that $f$ is not
the pull-back of a meromorphic function via $\pi$. In order to construct an
example for a semisimple group action from this one, let $T<G={\rm{SL}}_2(\C)$
be the maximal torus of diagonal matrices. Consider the diagonal action of $T$
on the product $G \times\C^2$, where $T\cong \C^*$ acts on $\C^2$ by scalar
multiplication as above. The group $G$ acts holomorphically on the quotient $X
:= G\times_T \C^2$ by this $T$-action. By construction we have $X \hq G
\cong \C^2\hq T =\{\text{point}\}$. Furthermore, the $G$-invariant rational
function on $X$ defined by $f\bigl[g,(z,w)\bigr] = z/w$ is not a pull-back via
the quotient map $X \to \{\text{point}\}$.
\end{ex}

The following result characterises those Stein $G$-spaces for which any
meromorphic invariant is a quotient of two holomorphic invariants, i.e., those
spaces for which $\pi^*$ is an isomorphism.

\begin{thm}\label{thm:aHquniversal}
Let $X$ be a $G$-irreducible Stein $G$-space with analytic Hilbert quotient
$\pi\colon  X \to X\hq G$. Then, the following are equivalent:
\begin{enumerate}
\item[a)]There exists a non-empty open subset $U \subset X\hq G$ such that for
all $q \in U$ the fibre $\pi^{-1}(q)$ contains a dense $G$-orbit.
\item[b)]There exists a non-empty Zariski-open subset $U \subset X\hq G$ such
that for all $q \in U$ the fibre $\pi^{-1}(q)$ contains a dense $G$-orbit.
\item[c)]The pull-back map $\pi^*$ establishes an isomorphism between
$\mathscr{M}_{X\hq G}(X\hq G)$ and $\mathscr{M}_X(X)^G$.
\end{enumerate}
\end{thm}

Before we prove Theorem~\ref{thm:aHquniversal} we list a few typical situations
where it can be applied.

\begin{cor}
Let $X$ be a $G$-irreducible Stein $G$-space with analytic Hilbert quotient
$\pi\colon X \to X\hq G$. Assume that there exists a point $x \in X$ such that
$\pi^{-1}\bigl(\pi(x)\bigr) = G \acts x$. Then, the pull-back $\pi^*$
establishes an isomorphism between $\mathscr{M}_{X\hq G}(X\hq G)$
and $\mathscr{M}_X(X)^G$.
\end{cor}

\begin{proof}
Under the hypotheses of the corollary, there exists a non-empty Zariski-open
subset $U \subset X\hq G$ such that for all $q \in U$ the fibre $\pi^{-1}(q)$
consists of a single $G$-orbit. Hence, Theorem~\ref{thm:aHquniversal} applies.
\end{proof}

\begin{cor}\label{cor:semisimple}
Let $G$ be a semisimple algebraic group and let $X$ be a $G$-irreducible
affine algebraic $G$-variety with factorial coordinate ring. Then,
$\mathscr{M}_{X\hq G}(X\hq G)$ and  $\mathscr{M}_X(X)^G$ are isomorphic via
$\pi^*$.
\end{cor}

\begin{rem}
The assumptions of Corollary~\ref{cor:semisimple} are in particular fulfilled
for $X = V$ a  representation space of a semisimple group.
\end{rem}

\begin{proof}[Proof of Corollary~\ref{cor:semisimple}]
Under the hypotheses of the corollary the generic fibre of the morphism
$\pi\colon  X \to \mathrm{Spec}\bigl(\C[X]^G\bigr)$ to the invariant-theoretic
quotient contains a dense $G$-orbit, see for
example~\cite[\S3.2]{PopovVinberg}. Since the corresponding map $\pi^h\colon
X^h \to \bigl(\mathrm{Spec}(\C[X]^G)\bigr)^h$ of complex spaces is the analytic
Hilbert quotient of the Stein space $X^h$, Theorem~\ref{thm:aHquniversal}
applies.
\end{proof}

In the remainder of the present
section we prove Theorem~\ref{thm:aHquniversal}.

\begin{proof}[Proof of Theorem~\ref{thm:aHquniversal}]
The implication b) $\Rightarrow$ a) is clear. As a second step we prove a)
$\Rightarrow$ b). Let $S_{\mathrm{max}} \subset X\hq G$ be the maximal
slice-type stratum, cf.~Section~\ref{subsect:strat}, and let $X_{\mathrm{max}}=
\pi^{-1}(S_{\mathrm{max}})$. Since $S_{\mathrm{max}}$ is dense in $X\hq G$,
there exists a point $q \in U \cap S_{\mathrm{max}}$. Recall from
Proposition~\ref{prop:Lunastrat} that the map $\pi|_{X_{\mathrm{max}}}$
realises $X_{\mathrm{max}}$ as a holomorphic fibre bundle over
$S_{\mathrm{max}}$ with typical fibre $\pi^{-1}(q)$. It therefore follows from
the assumption in a) that for every $q' \in S_{\mathrm{max}}$ the fibre
$\pi^{-1}(q')$ contains a dense $G$-orbit.

Next we prove c) $\Rightarrow$ b). Suppose on the contrary that the fibre
$\pi^{-1}(q)$ does not contain a dense $G$-orbit for any $q \in
S_{\mathrm{max}}$. Let $\Omega$ be a Zariski-open subset with geometric
quotient whose existence is guaranteed by the Main Theorem. We may assume that
$\Omega$ is contained in $X_{\mathrm{max}}$. Since $\Omega$ is Zariski-open and
dense in $X_{\mathrm{max}}$, for generic $q \in S_{\mathrm{max}}$ the
intersection $\Omega \cap \pi^{-1}(q)$ is Zariski-open and dense in
$\pi^{-1}(q)$. By assumption, this intersection therefore contains two distinct
$G$-orbits $G\acts x_1 \neq G\acts x_2$. By part (5) of the Main Theorem, there
exist an $f \in \mathscr{M}_X(X)^G$ whose values at $x_1$ and $x_2$ are
well-defined and distinct. Consequently, $f$ is not contained in
$\Image(\pi^*)$, a contradiction.

Finally, we prove implication b) $\Rightarrow$ c). We first study the local
geometry of the quotient map $\pi\colon  X \to X\hq G$ under the hypotheses of
b). As in the previous paragraph, let $\Omega$ be a Zariski-open subset with
geometric quotient $p\colon\Omega \to Q$. Without loss of generality we may
assume that $\Omega \subset X_{\mathrm{max}}$. Let $x_0 \in \Omega$ be chosen
such that $p(x_0)$ and $\pi(x_0)$ are smooth points of $Q$ and $X\hq G$,
respectively, and such that $G\acts x_0$ is dense in
$\pi^{-1}\bigl(\pi(x_0)\bigr)$. Since $\pi|_{X_{\mathrm{max}}}\colon
X_{\mathrm{max}} \to S_{\mathrm{max}}$ is a holomorphic fibre bundle there
exists a local holomorphic section $\sigma\colon U \to X$ of
$\pi|_{X_{\mathrm{max}}}$ through $x_0$, defined on a neighbourhood $U$ of
$\pi(x_0)$ in $X\hq G$. Since $x_0 \in \Omega$, we may assume that $\sigma (U)
\subset \Omega$. The situation is sketched in Figure~1.
\begin{figure}[h]
\includegraphics[width=9cm]{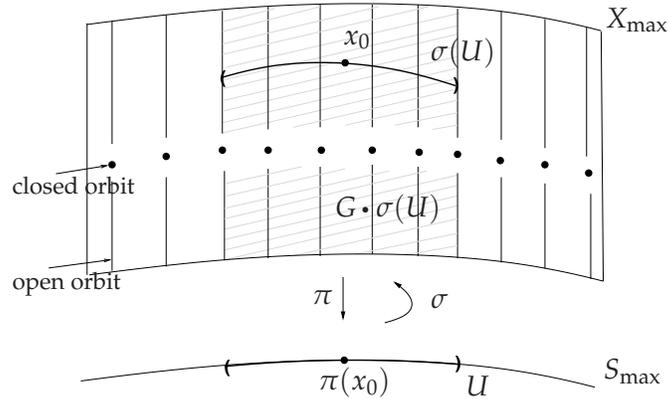}
\caption{The local geometry of the quotient map}
\end{figure}

Note that $\widetilde U := G\acts \sigma (U)\subset\Omega$ is open and
that $\pi|_{\widetilde U}\colon \widetilde U\to U$ parametrises the $G$-orbits
in $\widetilde U$ set-theoretically. Since $Q$ carries the quotient topology
with respect to $p$, the set $p(\widetilde U)$ is an open neighbourhood of
$p(x)$ in $Q$. By the choice of $x_0$, shrinking $U$ if necessary, we may
assume that both $U$ and $\widetilde U$ are smooth. Since $\pi$ is a
$G$-invariant holomorphic map, its restriction to $\Omega$ induces a uniquely
defined holomorphic map $\ol{\pi}\colon  Q \to X\hq G$ such that the diagram
\[\begin{xymatrix}{
X \ar[d]_{\pi} & \ar@{_{(}->}[l]\Omega \ar[d]^{p} \\
X\hq G & \ar[l]_>>>>>{\ol{\pi}}Q
} \end{xymatrix}
\]
commutes. By construction, the induced map $\ol{\pi}|_{p(\widetilde U)}\colon p
(\widetilde U) \to U$ is a holomorphic bijection between complex manifolds,
hence biholomorphic.

Let now $f \in \mathscr{M}_X(X)^G$ be given, and let $\Gamma_f \subset X \times
\P_1$ be its graph. It is our aim to show that $f$ descends to a meromorphic
function on $X\hq G$. The idea of the proof is to show that the image  $\pi
(P_f)$ of the pole variety $P_f$ of $f$ is nowhere dense in $X\hq G$ in order
to apply Proposition~5.2 of \cite{KaehlerQuotientsGIT}. To this end, we
consider the restricted functions $f|_{\widetilde U}\in
\mathscr{M}_X(\widetilde U)^G$ as well as $f|_{\pi^{-1}(U)}\in
\mathscr{M}_X(\pi^{-1}(U))^G$.

\begin{subclaim}There exists $\bar f \in \mathscr{M}_{X\hq G}(U)$ such that
$f|_{\pi^{-1}(U) } = \pi^*\bar f$.
\end{subclaim}

\begin{proof}[Proof of the subclaim]
Denote the graph of $f|_{\pi^{-1}(U)}$ by  $\Gamma'$, and set $\widetilde\Gamma
:= \Gamma' \cap (\widetilde U \times \P_1)$. Note that $\widetilde
\Gamma$ coincides with the graph of $f|_{\widetilde U}$. We define $\Pi
:= \pi|_{\pi^{-1}(U)} \times \id_{\P_1}\colon \pi^{-1}(U) \times \P_1
\to U \times \P_1$ and $P := p|_{\widetilde U} \times \id_{\P_1}\colon
\widetilde U \times \P_1 \to p(\widetilde U) \times \P_1$. The map $\Pi$
is an analytic Hilbert quotient, and $P$ is a geometric quotient for the
respective $G$-actions on $\pi^{-1}(U) \times \P_1$ and $\widetilde U \times
\P_1$. With this notation we summarise our setup in the following commutative
diagram
\[\begin{xymatrix}{
*{}\save[]+<0.6cm,0cm>*{\Gamma' \subset} \restore & \pi^{-1}(U) \times \P_1 \ar[d]_{\Pi} & \ar@{_{(}->}[l]
\ar[d]^{P}\widetilde U \times \P_1 & *{}\save[]-<1cm,0cm>*{\supset\widetilde \Gamma} \restore\\
  & U \times \P_1 & \ar[l]^{\cong}_{\ol{\Pi}}p(\widetilde U) \times \P_1.  &
}\end{xymatrix}
\]
The function $f|_{\widetilde U}$ descends to a meromorphic function on
$p(\widetilde U) \subset Q$ with graph $P (\widetilde \Gamma )$, cf.\ the proof
of part (1) of Proposition~\ref{Prop:PullbackMero}. Furthermore, we note that
$\ol{\Pi} \bigl(P(\widetilde \Gamma)\bigr)$ is analytic, hence closed in
$U\times\P_1$. From this and the fact that $\widetilde \Gamma$ is Zariski-open,
hence dense in the $G$-irreducible space $\Gamma'$ it follows that
\begin{equation*}
\Pi(\Gamma') = \overline{\Pi(\widetilde \Gamma)} = \overline{\ol{\Pi}
\bigl(P(\widetilde \Gamma)\bigr)}=\ol{\Pi}\bigl(P(\widetilde \Gamma)\bigr).
\end{equation*}
Here, $\overline{\;\,\cdot\, \; }$ denotes the topological closure in
$U\times\P_1$. Consequently, $\Pi(\Gamma')$ is a meromorphic graph over $U$
associated with a meromorphic function $\bar f \in \mathscr{M}_{X\hq G}(U)$
fulfilling $f|_{\pi^{-1}(U)} = \pi^* \bar f$.
\end{proof}

Finally, we consider the pole variety $P_f$ of $f$ and its image $\pi(P_f)$ in
$X\hq G$. Note that $\pi (P_f) \cap U$ coincides with the image under $\pi$ of
the pole variety of $f|_{\pi^{-1}(U)}$. The subclaim implies that
$f|_{\pi^{-1}(U)} = \pi^* \bar f$ for some $\bar f \in \mathscr{M}_{X\hq
G}(U)$. Consequently, $\pi (P_f) \cap U$ is nowhere dense in $U$. Since
$\pi(P_f)$ is analytic in the irreducible complex space $X\hq G$, we deduce
that $\pi(P_f)$ is nowhere dense in $X\hq G$. By Proposition~5.2 in
\cite{KaehlerQuotientsGIT}, this implies that the image of $\Gamma_f$ under the
map $\pi \times \id_{\P_1}\colon  X\times \P_1 \to X\hq G \times \P_1$ is a
meromorphic graph over $X\hq G$ and that $f$ descends to $X\hq G$. This
completes the proof of b) $\Rightarrow$ c).
\end{proof}

\subsection{Holomorphically convex $K$-spaces}\label{subsect:holconvex}

In this section we consider applications of the main result to spaces without
actions of complex groups, e.g. bounded domains.

\begin{prop}\label{prop:holconvex}
Let $X$ be a $K$-irreducible complex $K$-space of dimension $n$ and set
\begin{equation*}
m:= \max_{x \in X}\{\dim_\C \bigl(T_x(K\acts x) + J\cdot T_x(K \acts x)
\bigr)\}.
\end{equation*}
Suppose that there exist a Stein $K$-space $Y$ and an equivariant surjection
$\varphi\colon X\to Y$ which is injective outside a nowhere dense $K$-invariant
analytic set $A\subset X$. Then, there exist $(n-m)$ analytically independent
$K$-invariant meromorphic functions on $X$.
\end{prop}

\begin{rem}
Proposition~\ref{prop:holconvex} in particular applies to strongly pseudoconvex
(also called $1$-convex) complex spaces or, more generally, to holomorphically
convex spaces whose Remmert reduction is a proper modification.
\end{rem}

\begin{proof}[Proof of Proposition~\ref{prop:holconvex}]
Let $G=K^\mbb{C}$ be the complexification of $K$. According to the main result
of~\cite{HeinznerGIT} there exist a Stein $G$-space $Y^\mbb{C}$ and a
$K$-equivariant open embedding $\imath\colon Y\to Y^\mbb{C}$ such that
$Y^\mbb{C}=G\acts\imath(Y)$. Then, applying our Main Theorem we obtain a
$G$-invariant Zariski-open subset $\Omega\subset Y^\mbb{C}$ such that the
geometric quotient $p\colon\Omega\to Q=\Omega/G$ exists.

Since $\varphi$ is surjective and injective outside $A$, we have $\dim
Y^\mbb{C}=n$. Moreover, the maximal dimension of the $G$-orbits in $Y^\mbb{C}$
is $m$; hence, $\dim\Omega/G=n-m$. By part (5) of the Main Theorem the
$G$-invariant meromorphic functions separate the $G$-orbits in $\Omega$. This
implies there exist at least $n-m$ analytically independent $G$-invariant
meromorphic functions on $Y^\C$. Restricting these to $Y$ and pulling them
back to $X$ via $\varphi$ yields the desired $K$-invariant meromorphic
functions on $X$.
\end{proof}

\subsection{Actions with reductive generic stabiliser}\label{subsect:Richardson}

For $G$-connected Stein $G$-manifolds Richardson
\cite{RichardsonPrincipalSteinManifolds} proves the existence of a
\emph{principal orbit type} in the following sense: in every such manifold
there exists an open and dense subset $U$ such that the stabiliser groups of
points in $U$ are conjugate in $G$. Here, we sharpen his result in the case of
reductive stabiliser groups and draw a few consequences.

\begin{prop}\label{prop:genericreductivestab}
Let $G$ be a complex-reductive Lie group and let $X$ be a $G$-connected Stein
$G$-manifold. Assume that the principal orbit type is reductive. Then,
\begin{enumerate}
\item in the statement of the Main Theorem the set $\Omega$ can be chosen in
such a way that for all $x, y \in \Omega$ there exists a $g \in G$ with $G_y =
gG_xg^{-1}$. In particular, there exists a $G$-invariant Zariski-open dense
subset of $X$ consisting of orbits of principal orbit type;
\item additionally, $\Omega$ can be chosen such that $p\colon\Omega\to Q$ is an
analytic Hilbert quotient. In particular, $\Omega$ is Stein and $p\colon\Omega
\to Q$ is a holomorphic fibre bundle with typical fibre $G/G_x$.
\end{enumerate}
\end{prop}

\begin{rems}
(1) The assumption on the stabiliser groups is equivalent to the requirement
that the normaliser $N_G(G_x)$ be reductive, see \cite[proof of
Thm.~1.2(ii)]{ReichsteinStableModel}.

(2) The assumption on the stabiliser groups is automatically fulfilled by any
commutative reductive group acting on a connected Stein manifold.
\end{rems}

As a special case of Proposition~\ref{prop:genericreductivestab} we explicitly
note the result for the case of generically free actions.

\begin{cor}\label{cor:genericallyfree}
Let $G$ be a complex-reductive Lie group and let $X$ be a $G$-connected Stein
$G$-manifold. Assume that the action is generically free. Then, in the
statement of the Main Theorem the set $\Omega$ can be chosen such that $p\colon
\Omega\to Q$ is a $G$-principal fibre bundle.
\end{cor}

\begin{rem}\label{rem:torusaction}
The additional assumption of Corollary~\ref{cor:genericallyfree} is
automatically fulfilled by any commutative reductive group acting effectively
on a connected Stein manifold, as can be seen using the holomorphic Slice
Theorem.
\end{rem}

\begin{proof}[Proof of Proposition~\ref{prop:genericreductivestab}]
(1) Let $\phi\colon X \to V$ be a map from $X$ to a $G$-representation space
$V$ of the form guaranteed by the Weak Equivariant Embedding Theorem,
Proposition~\ref{prop:genericembedding}. Let $Y$ be the algebraic
Zariski-closure of $\phi (X)$ in $V$,
cf.~Section~\ref{subsect:geometricquotients}. Then, for any $y \in Y$, we
let $G_y = L_y \ltimes U_y$ be the Levi decomposition of the stabiliser $G_y$.
By a further result of Richardson
\cite[Thm.~9.3.1]{RichardsonDeformationLieSubgroups}, there exists a
$G$-invariant smooth Zariski-open subset $W$ of $Y$ such that $L_{y'}$ is
conjugate to $L_y$ in $G$ for all $y,y' \in W$ and such that $(U_y)_{y\in W}$
is an algebraic family of algebraic subgroups of $G$ in the sense of
\cite[Def.~6.2.1]{RichardsonDeformationLieSubgroups}. Owing to the definition
of $Y$, the image $\phi(X)$ intersects $W$ non-trivially. By the assumption on
the principal orbit type there exists a point $y_0 \in W \cap
\phi(X_\mathrm{max})$ such that $G_{y_0} = L_{y_0}$ and $U_{y_0}= \{e\}$. Since
the $(U_y)_{y\in W}$ form an algebraic family, we may assume that the number of
connected components of $U_y$ is constant for all $y \in W$. Consequently,
since $y_0 \in W$, the group $U_y$ is zero-dimensional and connected for all $y
\in W$, therefore trivial. It follows that $G_y$ is conjugate to $G_{y'}$ in
$G$ for all $y,y'\in W$. Consequently, the same is true for any pair of points
in the Zariski-open $G$-invariant subset $\phi^{-1}(W) \cap X_{\mathrm{max}}$
of $X$. Intersecting the set guaranteed by the Main Theorem with
$\phi^{-1}(W)\cap X_{\mathrm{max}}$ we arrive at the desired result.

(2) We have seen that in the affine algebraic variety $Y$ there exists a
Zariski-open and dense subset $W$ such that $G_x$ is reductive for all $x \in
W$. Hence, the existence of a Zariski-open subset $\Omega$ such that $p\colon \Omega
\to Q$ is an analytic Hilbert quotient is a consequence of
Lemma~\ref{lem:goodRosenlicht} below. The fact that the quotient map is a
holomorphic fibre bundle then follows by combining part (1) with the
holomorphic Slice Theorem, cf.~Proposition~\ref{prop:Lunastrat} and
\cite[Cor.~3.2.5]{LunaSlice}.
\end{proof}

\begin{lemma}\label{lem:goodRosenlicht}
Let $X$ be an algebraic $G$-variety. Assume that for $x\in X$ in general
position the stabiliser $G_x$ is reductive. Then, there exists a non-empty
Zariski-open affine $G$-invariant subset $U$ in $X$ such that the action of $G$
on $U$ admits a good geometric quotient. The associated map $U^h \to (U/G)^h$
of complex spaces is an analytic Hilbert quotient.
\end{lemma}

\begin{proof}
By a result of
Reichstein and Vonessen \cite{ReichsteinStableModel} there exists a birational
$G$-equivariant map $\phi\colon X \dasharrow Y$ to an affine $G$-variety $Y$
with the following property: if $\pi_Y\colon Y \to Y\hq G$ denotes the
categorical good quotient, then the set $Y^{st} := \{y \in Y \mid
\pi_Y^{-1}\bigl(\pi_Y(y)\bigr)= G \acts y\}$ is non-empty. Without loss of
generality, we may assume that $Y^{st}$ is affine and that
$\phi^{-1}|_{Y^{st}}\colon Y^{st}\to X$ is an isomorphism onto its image. This
shows the first claim. The second claim is a consequence of
\cite{Lunaalgebraicanalytic}.
\end{proof}

\subsection{Unipotent groups}\label{subsect:unipotent}

In this section we discuss consequences of our main result for actions of
unipotent groups.

\begin{prop}
Let $H<G$ be a unipotent algebraic subgroup of a complex-reductive Lie group
and  let $X$ be a Stein $G$-space. Then, every $H$-orbit in $X$ is closed.
Furthermore, the topological quotient $X/H$ is generically Hausdorff and there
exists a $H$-invariant Zariski-open dense subset $U$ of $X$ such that the
restriction of the topological quotient $p\colon X\to X/H$ to $U$ is a
geometric quotient in the category of complex spaces.
\end{prop}

\begin{proof}
Assuming that every $H$-orbit is closed, the remaining statements follow
directly from the Main Theorem. So, let $H\acts x \subset X$ be any $H$-orbit.
Then, clearly $\overline {H\acts x} \subset \overline{G\acts x}$. However,
the fibre $\pi^{-1}\bigl(\pi(x)\bigr) \supset \overline{G\acts x}$ of the
analytic Hilbert quotient $\pi\colon X \to X\hq G$ carries a natural affine
algebraic structure with respect to which the $G$-action is algebraic, see
\cite[Cor.~5.6]{Snow}. Since $H < G$ is algebraic by assumption, by the
corresponding result in the affine algebraic case (which is proven for example
in \cite[Appendix]{BirkesOrbits}) the orbit $H\acts x$ is closed in
$\pi^{-1}\bigl(\pi(x)\bigr)$ and hence in $X$.
\end{proof}

\begin{ex}
There exists a domain of holomorphy $D$ in $\mbb{C}^2$ endowed with a free
holomorphic action of $H=\mbb{C}$ such that the topological closure of every
$H$-orbit is a real hypersurface in $D$, cf.~\cite[Sect.~7-8]{HubbOV}. In
particular, there is no open $\C$-invariant subset $\Omega$ in $D$ such that
$\Omega/\C$ is Hausdorff. Hence, we cannot expect holomorphic actions of
algebraic groups to have any of the properties stated in the Main Theorem if
they do not extend to holomorphic actions of some complex-reductive group.
\end{ex}

\section{Examples}

In the analytic setup the question of existence of a Rosenlicht quotient
consists of the following two parts:
\begin{enumerate}
\item Does there exist a Zariski-open subset on which the action of $G$ admits
a geometric quotient?
\item Do the invariant meromorphic functions separate the $G$-orbits in general
position?
\end{enumerate}
In the following we are going to describe examples showing that the assumption
made in the Main Theorem are indeed necessary to obtain a positive answer to
both questions.

In the Main Theorem it is assumed that the group $H$ under discussion is an
algebraic subgroup of a reductive group acting on $X$. The following examples
show that this algebraicity assumption is indeed necessary in order to obtain a
geometric quotient. All these examples deal with actions of discrete groups,
which we denote by $\Gamma$ instead of $H$.

\begin{ex}
Let $G=\C^*\times\C^*$ and let $\Gamma \cong \Z^2$ be the discrete subgroup
generated by the elements $e$ and $e^{-\pi}$. Furthermore, consider the
subgroup $M = \textrm{diag}(\C^*)$. Then, the quotient $X=G/M$ is isomorphic to
$\C^*$ via the map $[(z,w)]\mapsto z/w$, and the induced $\Z^2$-action is given
by $(m_1, m_2)\acts z = e^{m_1 + m_2\pi}$. For this action there does not even
exist an open subset of $\C^*$ that admits a Hausdorff topological quotient.
\end{ex}

Note that the action of the ambient reductive group $G$ is not effective in
the above example. The next example shows that even if the $G$-action is
effective geometric quotients for non-algebraic subgroups might not exist.

\begin{ex}
We consider the action of the discrete subgroup $\Gamma:={\rm{SL}}_2(\Z)$ of
$G:={\rm{SL}}_2(\C)$ on the homogeneous Stein manifold $X=G/T$, where $T$ is
the maximal torus of diagonal matrices in $G$. We claim that there does not
exist a $\Gamma$-invariant Zariski-open subset $U$ of $X$ such that the
quotient $U/\Gamma$ is Hausdorff. First, we consider an explicit realisation of
this action. Taking a regular element $\xi\in\mathfrak{t} = \Lie (T)$ we have
$\Ad(G)\xi\cong G/T$. Note that the ring of invariants
$\mathbb{C}[\mathfrak{g}]^G$ for the adjoint action of $G$ on its Lie algebra
$\mathfrak{g}$ is equal to $\mathbb{C}[\det]$. An element $\xi\in\mathfrak{g}$
is regular if and only if $\det(\xi)\not=0$. Therefore, $G/T$ can be identified
with
\begin{equation*}
\left\{(x,y,z)\in\mathbb{C}^3\,;\, \det\begin{pmatrix}x&y\\z&-x\end{pmatrix}=-1
\right\}.
\end{equation*}
The $\Gamma$-action on $X=G/T$ in this realization is induced by conjugation.
As an auxiliary tool, we are going to consider the induced action of $\C$ and
of the discrete subgroup $\Z < \C$ on $X$ given by the embedding of $\C$ into
${\rm{SL}}_2(\C)$ as upper triangular matrices. Explicitly,
for $t\in \C$ we obtain
\begin{equation}\label{eq:Caction}
t\acts(x,y,z)=(x+tz,y-2tx-t^2z,z).
\end{equation}
Let now $U$ be any $\Gamma$-invariant Zariski-open subset of $X$. Since there
are no $\Gamma$-invariant analytic hypersurfaces in $G$ (see
\cite{AkhiezerInvariantMeromorphic}, or \cite[\S2.2, Ex.2]{HuckOeljeBook} for
an elementary proof) there are no such hypersurfaces in $X$ either.
Consequently, the complement $A:= X \setminus U$ has pure codimension two in
$X$, i.e., $A$ is a discrete set of points. The map $p\colon X\to\mathbb{C}$,
$p(x,y,z)=z$ is $\mathbb{C}$-invariant, the fibre $p^{-1}(a)$ for $a \neq 0$
consists of a single orbit. However, the fiber $p^{-1}(0)$ is the union of the
two $\mathbb{C}$-orbits $\mathbb{C}\acts (1,0,0) = \{(1,y,0)\,;\, y \in \C \}$
and $\mathbb{C}\acts (-1,0,0)= \{(-1, y, 0) \,;\, y \in \C \}$, which cannot be
separated by $\C$-invariant open neighbourhoods. Hence, $X/\mathbb{C}$ is not
Hausdorff.

We show by direct calculation that $U/\Gamma$ is not Hausdorff, either. Since
$A$ consists of isolated points, there exists $y_0 \in \C \setminus
\mathbb{Q}$ such that both $(1, y_0, 0)$ and $(-1,y_0,0)$ are contained in $U$.
Using that $y_0$ is irrational, one checks by direct computation that these two
points lie in different $\Gamma$-orbits. We are going to show that the orbits
$\Gamma\acts(1,y_0,0)$ and $\Gamma \acts (-1, y_0, 0)$ cannot be separated by
invariant open sets. To this end, let $V$ be any open $\Gamma$-invariant
neighbourhood of $(1, y_0, 0)$ in $U$. Since $A$ is discrete, for all integers
$m\gg0$ the points $p_m:=(1,y_0,-2/m)$ are contained in $V$. Using
\eqref{eq:Caction} we compute
\begin{equation*}
m\acts p_m  = \left(1+m\cdot\frac{-2}{m},y_0-2m-m^2\cdot\frac{-2}{m},
\frac{-2}{m}\right) = \left(-1, y_0, \frac{-2}{m}\right) \in V.
\end{equation*}
If $W$ is any open neighbourhood of $(-1,y_0,0)$ in $U$, then by the above
computation $m\acts p_m \in W$ for $m \gg0$. Since the $\Z$-orbits of the
$p_m$ are contained in the corresponding $\Gamma$-orbits, every
$\Gamma$-invariant neighbourhood $V$ of $(1, y_0, 0)$ thus intersects every
open neighbourhood $W$ of $(-1,y_0,0)$, so $U/\Gamma$ cannot be Hausdorff.

As a concluding remark, note that by removing the fibre $p^{-1}(0)$ from $X$ we
obtain a Zariski-open $\C$-invariant subset of $X$ on which the $\C$-action
admits a Hausdorff quotient, in accordance with Rosenlicht's Theorem and with
the main result of this paper.
\end{ex}

In the case of non-algebraic subgroups acting on Stein manifolds the invariant
meromorphic functions do not necessarily separate generic orbits, even if a
meromorphic quotient exists. This is exemplified in the following.

\begin{ex}
We consider the subgroup $\Gamma:= {\rm{SL}}_2(\Z)$ of ${\rm{SL}}_2(\C)$ acting
on the Stein manifold $X={\rm{SL}}_2(\C)$ by $\gamma\acts g = g\gamma^{-1}$.
Then, the action is proper and free, and hence the geometric quotient $X/\Gamma$
exists.  Let $U$ be any $\Gamma$-invariant analytically Zariski-open subset of
$X$. Then $U/\Gamma$ exists and is biholomorphic to the image of $U$ in the
quotient $X/\Gamma$. Using that there are no $\Gamma$-invariant hypersurfaces
in ${\rm{SL}}_2(\C)$ we see that the complement of $U$ in $X$ has no
codimension-one components. It follows that every $\Gamma$-invariant
meromorphic function on $U$ extends to a $\Gamma$-invariant meromorphic
function on the whole of $X$ by Levi's Theorem. However, because of the
non-existence of invariant hypersurfaces, the pole variety of every
$\Gamma$-invariant meromorphic function on $X$ is empty. Hence, every such
function is holomorphic, and therefore constant.
\end{ex}

\section{A weak equivariant embedding theorem}

In the following let $X$ be a $G$-irreducible Stein $G$-space for a
complex-reductive Lie group $G$. The main technical ingredient in the proof of
our main result is an equivariant version of the following result of Remmert
and Narasimhan:

\begin{thm}[\cite{NarasimhanEmbedding}]\label{thm:NarasimhanEmbedding}
Let $X$ be a finite-dimensional Stein space. Then there exist a finite
dimensional complex vector space $V$ and a proper injective holomorphic map
$\phi\colon X\to V$ that is an immersion on $X\setminus X_{\text{sing}}$.
\end{thm}

In this paper a holomorphic map $\phi\colon X\to Y$ from a Stein space $X$ into
an affine variety $Y$ is called a \emph{Narasim\-han map} if $\phi$ is proper
injective and if $\phi|_{X\setminus X_\text{sing}}$ is an immersion.

In the folllowing we will investigate to what extend there exists an
equivariant version of this fundamental result.

Suppose that the Stein $G$-space $X$ admits an equivariant Nara\-sim\-han map
$\phi\colon X\to V$ into a finite-dimensional $G$-representation space $V$.
In this situation the stratification of $X\hq G$ into orbit-types is necessarily
finite. In contrast, Heinzner~\cite[Sect.\ 3]{HeinznerEinbettungen} has given
an example of a Stein $\C^*$-manifold that contains a sequence of points
$\{x_n\}$ lying in closed $\C^*$-orbits and having iso\-tro\-py groups
$\C^*_{x_n}=\Z _{p_n}$, where $\{p_n\}$ is a sequence of prime numbers such
that $\lim_{n\to \infty}p_n =\infty$. Hence, the first guess for an equivariant
version of Theorem~\ref{thm:NarasimhanEmbedding} does not lead to the desired
result.

The problems encountered in the above example are caused be the appearance of
''too many'' different isotropy groups and slice-representations for our given
action. However, recall from~Proposition~\ref{prop:Lunastrat} that there exists
a maximal, Zariski-open stratum $S_{\mathrm{max}}$ in $X\hq G$ over which the
type of the slice representation is constant. Using the methods of
\cite{HeinznerEinbettungen} we prove the following equivariant version of
Theorem~\ref{thm:NarasimhanEmbedding}.

\begin{prop}[Weak Equivariant Embedding Theorem]\label{prop:genericembedding}
Let $X$ be a $G$-irreducible Stein $G$-space with associated analytic Hilbert
quotient $\pi\colon~X \to X\hq G$. Let $S_{\mathrm{max}}$ be the maximal
stratum of the slice-type stratification, and
$X_{\mathrm{max}}:=\pi^{-1}(S_{\mathrm{max}})$. Then there exists a
finite-dimensional $G$-module $V$ with analytic Hilbert quotient $\pi_V\colon V
\to V\hq G$, and a holomorphic $G$-equivariant map $\phi\colon X\to V$ with the
following properties:
\begin{enumerate}
\item The induced holomorphic map $\ol{\phi}\colon X\hq G\to V\hq G$ is a Narasimhan map
and
\item the induced holomorphic map
\begin{equation*}
\phi|_{X_{\mathrm{max}}}\colon X_{\mathrm{max}}\to V_{\mathrm{max}}:=V\setminus
\pi_V^{-1}\bigl(\ol{\phi}(S_{\mathrm{max}}^{c})\bigr)
\end{equation*}
is a closed embedding.
\end{enumerate}
\end{prop}

The main technical part in the proof of Proposition~\ref{prop:genericembedding}
is contained in the following lemma, the proof of which is adapted
from~\cite[Sect.4, Lemma~1]{HeinznerEinbettungen}. For the reader's convenience
we describe the arguments here in some detail. In the following we write
$A_{\mathrm{max}} :=A\cap S_{\mathrm{max}}$ for any analytic subset $A\subset
X\hq G$.

\begin{lemma}\label{lem:improvingmaps}
Let $X$ be a $G$-irreducible Stein $G$-space and let $\pi\colon X\to X\hq G$
be its analytic Hilbert quotient. For any analytic subset $A$ of $X\hq G$ that
intersects the maximal slice-type stratum $S_{\mathrm{max}}$ non-trivially the
following holds:
\begin{enumerate}
\item There exists an analytic subset $A'$ of $A$ with
$\dim(A'_{\mathrm{max}})<\dim(A_{\mathrm{max}})$, a complex $G$-module $V_1$,
and an equivariant holomorphic map $\phi\colon X\to V_1$ that is an immersion
along $\pi^{-1}(A_{\mathrm{max}}\setminus A'_{\mathrm{max}})$.
\item There exists an analytic subset $A''$ of $A$ with
$\dim(A''_{\mathrm{max}})<\dim(A_{\mathrm{max}})$, a complex $G$-module $V_2$,
and an equivariant holomorphic map $\psi\colon X\to V_2$ whose restriction to
every closed $G$-orbit in $\pi^{-1}(A_{\mathrm{max}}\setminus
A''_{\mathrm{max}})$ is a proper embedding.
\end{enumerate}
\end{lemma}

\begin{proof}
Let $A\subset X\hq G$ be an analytic subset such that
$A_{\mathrm{max}}\not=\emptyset$. Removing the irreducible components which are
not of maximal dimension, we may assume without loss of generality that
$A_{\mathrm{max}}$ is pure-dimen\-sional. We denote the irreducible components
of $A_{\mathrm{max}}$ by $A_i$, $i\in I$, and choose for every $i\in I$ a point
$p_i\in A_i\setminus\bigcup_{j\neq i}A_j$. For each $i$, let $x_i\in
\pi^{-1}(p_i)$ be a point lying in the unique closed $G$-orbit in this fibre.

The slice type of every point $x_i$, $i\in I$, is equal to a fixed model
$G\times_H W$. This model admits an equivariant holomorphic embedding into a
$G$-module $V_1$. By the holomorphic Slice Theorem (\cite{Snow}), for each
$i\in I$ we can choose a small neighbourhood $U_i$ of $p_i$ in $X$ such that
$\pi^{-1}(U_i)$ has a $G$-equivariant holomorphic embedding into a saturated
open subset of the $G$-module $V_1$. Application of \cite[Sect.~1,
Prop.~1]{HeinznerEinbettungen} to the induced map $\dot\bigcup_{i\in
I}\pi^{-1}(U_i)\to V_1$ yields a $G$-equivariant holomorphic map $\phi\colon
X\to V_1$ that is an immersion along $\dot\bigcup_{i\in I}\pi^{-1}(p_i)$. The
set
\begin{equation*}
R:=\{x\in X \mid \phi\text{ is not an immersion in }x\}
\end{equation*}
is a $G$-invariant analytic subset of $X$. It follows that $A':=\pi(R)\cap A$
is an analytic subset of $A$. Since the map $\phi$ is an immersion at every
point in $\bigcup_{i \in I}\pi^{-1}(p_i)$, we conclude that
$\dim(A'_{\mathrm{max}})<\dim(A_{\mathrm{max}})$, as desired.

Let us now prove the second claim. Since by definition the slice type is
constant on $S_{\mathrm{max}}$, all orbits $G\acts x_i$ have the same orbit
type $(H)$, where $H$ is a complex-reductive subgroup of $G$. Since there is a
proper equivariant embedding of $G/H$ into some $G$-module $V_2$, there exists
a proper holomorphic map $\dot\bigcup_{i\in I}G\acts x_i\to V_2$. By
\cite[Sect.~1, Bemerkung~2]{HeinznerEinbettungen} this map extends to
a $G$-equivariant holomorphic map $\psi\colon X\to V_2$ such that $\psi|_{G
\acts x_i}$ is a proper embedding for each $i\in I$.

Let $\pi_{V_2}\colon V_2\to V_2\hq G$ be the analytic Hilbert quotient. For
$q\in V_2\hq G$ we set $\Type(q):=\Type(G\acts q)$ and define
\begin{equation*}
C:=\bigl\{q\in V_2\hq G\mid\Type(q)<H\bigr\} \subset V_2\hq G.
\end{equation*}
Then $\Omega:=V_2\setminus\pi_{V_2}^{-1}(C)$ is an algebraically Zariski-open
$G$-saturated subset of $V_2$, see~\cite[Ch.~III]{LunaSlice}. Note that we
have $\psi(x_i)\in \Omega$ for all $i\in I$. Let
$\ol{\psi}\colon X\hq G\to V_2\hq G$ be the induced map and set
$A'':=\ol{\psi}^{-1}(C)$, which is an analytic subset of $X\hq G$. We want to
show that for every $x\in X$ such that $G\acts x$ is closed with $\pi(x)\in
A_{\mathrm{max}}\setminus A''_{\mathrm{max}}$ the restriction of $\psi$ to
$G\acts x$ is a proper embedding into $V_2$.

For this suppose that $\psi(G\acts x)=G\acts\psi(x)$ is not closed. Then there
is a unique closed $G$-orbit $G\acts v\subset\Omega$ in the closure of
$G\acts\psi(x)$ and for this orbit we have $\Type(G\acts v)<\Type\bigl(G\acts
\psi(x)\bigr)\leq\Type(G\acts x)=(H)$, a contradiction. Consequently, $G\acts
\psi(x)$ must be closed in $V$ and, since it lies in $\Omega$, we have $\Type
\bigl(G\acts\psi(x)\bigr)=(H)$. Therefore, $\psi\colon G\acts x\to G\acts
\psi(x)$ is an isomorphism, hence $\psi|_{G\acts x}$ is a proper embedding.
Finally, since all the $x_i$ are contained in $A_{\mathrm{max}}\setminus
A''_{\mathrm{max}}$, clearly $\dim(A''_{\mathrm{max}})<\dim(A_{\mathrm{max}})$,
as claimed.
\end{proof}

Now we are in the position to give the proof of
Proposition~\ref{prop:genericembedding}.

\begin{proof}[Proof of Proposition~\ref{prop:genericembedding}]
Let $X$ be a $G$-irreducible Stein $G$-space with associated analytic Hilbert
quotient $\pi\colon X\to X\hq G$. Let $\varphi_0\colon X\hq G\to V_0$ be a
Narasimhan map and $\phi_0\colon X\to V_0$ the lifted map $\phi_0:=
\varphi_0\circ \pi$. By a repeated application of the first part of
Lemma~\ref{lem:improvingmaps} we obtain an equivariant holomorphic map $\phi_1
\colon X\to V_1$ to a complex $G$-module $V_1$ that is an immersion at every
point in $\pi^{-1}(S_{\mathrm{max}})$. Additionally, by a repeated application
of the second part of the same lemma we obtain an equivariant holomorphic map
$\psi\colon X\to V_2$ into a complex $G$-module whose restriction to every
closed orbit in $\pi^{-1}(S_{\mathrm{max}})$ is a closed embedding. Let
$V:=V_0 \oplus V_1\oplus V_2$ and let $\phi\colon X \to V$ be the product map.

Let $\pi_V\colon V\to V\hq G$ denote the quotient by the $G$-action and let
$\ol{\phi}\colon X\hq G\to V\hq G$ be the induced map. Since $\ol{\phi}$ is
proper (we started with $\varphi_0$ which was assumed to be a Narasimhan map),
the image of $S_{\mathrm{max}}^c:=(X\hq G)\setminus S_{\mathrm{max}}$ under
$\ol{\phi}$ is an analytic subset of $V\hq G$. The restriction
$\phi|_{X_\mathrm{max}}\colon X_{\mathrm{max}} \to V \setminus
\pi_V^{-1}(\ol{\phi}(S_{\mathrm{max}}^c)) =: V_{\mathrm{max}}$ is an immersion
and a closed embedding when restricted to any closed orbit in
$X_{\mathrm{max}}$. By \cite[Sect.~2, Prop.~2]{HeinznerEinbettungen} the
restriction of $\phi$ to every fibre of $\pi$ is a closed embedding. Hence,
$\phi$ is an injective immersion, since $\ol{\phi}$ separates the points of
$S_{\mathrm{max}}$. It therefore remains to check that
$\phi|_{X_{\mathrm{max}}}$ is proper, which can be done the same way as in the
last paragraph in the proof of \cite[Thm.~9.6]{PaHq}.
\end{proof}

\section{Constructing geometric quotients}

We continue to consider the action of a complex-reductive group $G$ on a Stein
space $X$ as well as the induced action of an algebraic subgroup $H$ of $G$. In
this section we prove the existence of an $H$-invariant Zariski-open dense
subset $\Omega$ of $X$ that admits a geometric quotient $p\colon \Omega\to Q$
with the properties listed in parts (1) -- (3) of the Main Theorem.

The idea of proof is to use the Weak Equivariant Embedding Theorem established
above in order to reduce to an algebraic situation. Then classical results
on algebraic transformation groups and especially Rosenlicht's theorem will
allow us to show the existence of geometric quotients for algebraic subgroups
$H\subset G$.

In order to avoid the corresponding technical difficulties we show in
Section~\ref{Subs:H-irr} that it is sufficient to treat the $H$-irreducible
case. In Section~\ref{Subs:UnivGeomQuot} we then discuss the universality
properties of (algebraic) geometric quotients before we prove the existence of
geometric $H$-quotients in the final subsection.

\subsection{Reduction to $H$-irreducible Stein $G$-spaces}\label{Subs:H-irr}

Let $X$ be a Stein $G$-space, and $H$ an algebraic subgroup of $G$. Suppose
that the Main Theorem is proven under the additional asumption that $X$ is
$H$-irreducible.

Let $X=\bigcup_{i=1}^mX_i$ be the decomposition of $X$ into its $H$-irreducible
components. Then, we may apply the Main Theorem to each of the components $X_i$
and obtain $H$-invariant Zariski-open dense subsets $\Omega_i\subset X_i$ with
geometric quotients $\Omega_i\to\Omega_i/H$. Note that we can choose $\Omega_i$
to be contained in $X_i \setminus \bigcup_{k\neq i}X_k$. It follows that the
disjoint union of the sets $\Omega_i$ is $H$-invariant, Zariski-open and dense
in $X$. Furthermore, it admits a geometric quotient
$\dot\bigcup_{i=1}^m\Omega_i\to\dot\bigcup_{i=1}^m(\Omega_i/H)$ by the
$H$-action with the properties listed in the Main Theorem.

\subsection{Universality of geometric quotients}\label{Subs:UnivGeomQuot}

First, we discuss the universality properties of geometric quotients with
respect to invariant analytic subsets.

\begin{lemma}\label{lem:imageundergeomquotmap}
Let $L$ be a complex Lie group and let $X$ be an $L$-irreducible holomorphic
$L$-space admitting a geometric quotient $p\colon X\to X/L$. If $A\subset X$ is
an $L$-invariant analytic subset of $X$, then $p(A)$ is an analytic subset of
$X/L$.
\end{lemma}

\begin{proof}
Since the geometric quotient $p\colon X\to X/L$ exists, we conclude from
\cite[\S3, Satz~7]{HolmannSlice} that all $L$-orbits are analytic in $X$
and have the same dimension. Hence, the corollary in Section~3.7
of~\cite{Fischer} applies to show that $p(A)$ is locally analytic in $X/L$.
Since the image of an $L$-invariant closed set under $p$ is again closed, the
set $p(A)$ is analytic in $X/L$.
\end{proof}

More can be said if the quotient map $p\colon X\to X/L$ is assumed to be a
submersion:

\begin{lemma}\label{Lem:UniversalityofGeomQuot}
Let $L$ be a complex Lie group and let $X$ be an $L$-irreducible holomorphic
$L$-space admitting a geometric quotient $p\colon X\to X/L$. Suppose that $p$
is a submersion. Then, $p$ is universal with respect to $L$-invariant analytic
sets of $X$.
\end{lemma}

\begin{proof}
Given an $L$-invariant analytic subset $A$ of $X$, we must show that the map
$p|_A\colon A \to p(A)$ fulfills properties (1) -- (3) of
Definition~\ref{defi:geometricquotient}. We already know from
Lemma~\ref{lem:imageundergeomquotmap} that $p(A)$ is an analytic subset of
$X/L$. The fibres of $p|_A$ are $G$-orbits, since the same is true for $p$.
Hence, it remains to show that $p(A)$ carries the quotient topology with
respect to $p|_A$ and that the structure sheaf of $p(A)$ as a reduced complex
subspace of $X/L$ is isomorphic to the sheaf of $L$-invariant holomorphic
functions on $A$.

Using that $X/L$ carries the quotient topology with respect to $p$ one checks
directly that the same is true for $A$ and $p|_A$: Let $U \subset A$ be an
$L$-invariant open subset. By definition of the subspace topology on $A$, there
exists an open subset $\widetilde U$ in $X$ such that $\widetilde U \cap A =U$.
Then, $\widehat U := G\acts \widetilde U$ is $G$-invariant, open, and
still fulfills $\widehat U \cap A =U$. Consequently, $p(\widehat U)$ is open in
$X/L$ and we have $p|_A(U) = p(\widehat U) \cap p(A)$. Hence, $p|_A(U)$ is open
in $p(A)$.

It remains to consider the structure sheaves. For this let $U\subset A$ be an
$L$-invariant open subset and let $f\in\mathscr{O}(U)^L$. Then there is a
continuous function $\bar f\colon p(U)\to\mbb{C}$ with $f=p^*\bar f$ and we
must show $\bar f\in\mathscr{O}_{p(A)}\bigl(p(U)\bigr)$. Since this assertion
is local, we may apply \cite[Thm.\ in \S2.18]{Fischer} to $p$ and obtain, after
possibly shrinking $U$, a commutative diagram
\begin{equation*}
\xymatrix{
\wt{U}\ar[rr]^\psi\ar[dr]_p & & D\times \wt{V}\ar[dl]^{\pi_{\wt{V}}}\\
 & \wt{V}, &
}
\end{equation*}
where $\wt{U}$ is an open subset of $X$ such that $U=\wt{U}\cap A$, where
$\wt{V}$ is an open subset of $X/L$ such that $p(\wt{U})\subset\wt{V}$, and
where $D$ is a domain in $\mbb{C}^N$ for $N=\dim X-\dim X/L$. In this picture
we have $\psi(A)=D\times p(A)$ and $f\circ\psi^{-1}(z,y)=\bar f(y)$ for $y\in
p(A)$. Hence, $\bar f$ is indeed holomorphic on $p(U)$, as was to be shown.
\end{proof}

\begin{rem}
Lemma~\ref{Lem:UniversalityofGeomQuot} can be used to prove the following
observation which might be of independent interest: Let $H$ be an algebraic
group and let $X$ be an algebraic $H$-irreducible $H$-variety admitting an
(algebraic) geometric quotient $p\colon X\to X/H$ such that $p$ is a
submersion. If $A\subset X$ is an $H$-invariant analytic subset, then
$p|_A\colon A\to p(A)$ is a (holomorphic) geometric $H$-quotient.
\end{rem}

\subsection{Existence of geometric quotients for algebraic subgroups of $G$}
\label{subsect:geometricquotients}

Let $G$ be complex-reductive and let $X$ be an $H$-irreducible Stein
$G$-space where $H$ is an algebraic subgroup of $G$. Note that $X$ is also
$G$-irreducible. We want to combine the results of the previous two sections in
order to prove that geometric quotients for the $H$-action exist on
Zariski-open dense subsets of $X$.

We introduce some notation in order to prepare the proof of statements
(1)--(3) of the Main Theorem. Let $\phi\colon X\to V$ be the weak equivariant
embedding constructed in Proposition~\ref{prop:genericembedding} and let
$S_{\mathrm{max}}$ be the maximal slice-type stratum in $X\hq G$. Since
the induced map $\ol{\phi}\colon X\hq G\to V\hq G$ is in particular proper,
$\ol{\phi}(S_\mathrm{max}^c)$ is an analytic subset of $V\hq G$ and the
set $V_{\mathrm{max}}=V\setminus\pi_V^{-1}\bigl(\ol{\phi}(S_{\mathrm{
max}}^c)\bigr)$ is analytically Zariski-open in $V$. Let $Y$ be the algebraic
Zariski-closure of $\phi(X)$ in $V$. The preimage $X_{\mathrm{max}}=
\pi^{-1}(S_{\mathrm{max}})$ is analytically Zariski-dense in $X$, hence
the algebraic Zariski-closure of $\phi(X_{\mathrm{max}})$ coincides with $Y$.
Note furthermore that $Y$ is $H$-irreducible.

In summary, we have found a $G$-equivariant map $\phi\colon X\to Y$ into an
$H$-irreducible affine variety $Y$ which is a proper embedding from
$X_{\mathrm{max}}$ into $Y_{\mathrm{max}}:=Y\cap V_{\mathrm{max}}$. Since the
$G$-action on $Y$ is algebraic we may now apply classical results on algebraic
transformation groups and transport them to $X$ via $\phi$.

By Rosenlicht's Theorem, see Theorem~\ref{thm:Rosenlicht}, there exists an
algebraically Zariski-open $H$-irreducible subset $\Omega_Y$ of $Y$ that admits
an algebraic geometric quotient $p_Y\colon\Omega_Y\to\Omega_Y/H$ by the
$H$-action. In the next step we will shrink $\Omega_Y$ in order to improve the
properties of $p_Y$ and of $\Omega_Y/H$. Note that the set where $p_Y$ is a
submersion is an $H$-invariant algebraically Zariski-open subset of $\Omega_Y$;
shrinking $\Omega_Y$ we may assume that $p_Y$ is a submersion. Repeating this
procedure if necessary we may also assume that $\Omega_Y$ and $\Omega_Y/H$ are
smooth and that $\Omega_Y/H$ is affine. Finally, using~\cite[Cor.~5.1]{Verdier}
we may furthermore suppose that $p_Y\colon\Omega_Y\to\Omega_Y/H$ is a
topological fibre bundle with respect to the complex topologies of $\Omega_Y$
and $\Omega_Y/H$. Note that we still have
$\phi(X_{\mathrm{max}})\cap\Omega_Y\not=\emptyset$ since $\Omega_Y$ is
algebraically Zariski-dense in $Y$.

After these preparations we are now in the position to prove
the existence of a geometric $H$-quotient on a dense Zariski-open subset of
$X$:

\begin{prop}\label{Prop:ExistenceGeomQuot}
Let $X$ be an $H$-irreducible Stein $G$-space where $H$ is an algebraic
subgroup of $G$. Then there exist a Zariski-open dense $H$-invariant subset
$\Omega$ of $X$ and a holomorphic map $p\colon\Omega\to Q$ to an irreducible
Stein space $Q$ that is a geometric quotient for the $H$-action on $\Omega$ and
additionally possesses the properties listed under (2) and (3) in the Main
Theorem.
\end{prop}

\begin{proof}
We use the notation introduced above. We first define the desired set $\Omega$.
To this end, let $\pi\colon Y\to Y\hq G$ denote the Hilbert quotient of
the $G$-action on $Y$ and note that, since $Y\hq G$ is an affine variety and
thus a Stein space, we may find a non-constant function $f\in\mathscr{O}_{Y\hq
G}(Y\hq G)$ which vanishes on $\ol{\phi}(S_\mathrm{max}^c)$. Consequently,
$(Y\hq G)\setminus\{f=0\}$ is an analytically Zariski-open Stein subset of
$Y\hq G$. Let $U$ be its inverse image under $\pi$ in $Y$ and define
\begin{equation*}
\Omega:=\varphi^{-1}(U\cap\Omega_Y)\subset X.
\end{equation*}
By construction, $\Omega$ is an $H$-irreducible analytically Zariski-open dense
subset of $X$ contained in $X_\mathrm{max}$. Moreover, we define
\begin{equation}\label{Eqn:DefQuotMap}
p:=p_Y\circ(\phi|_\Omega)\colon\Omega\to Q:=p_Y\bigl(\phi(\Omega)\bigr)
\subset\Omega_Y/H.
\end{equation}
Since $\phi\colon X_\mathrm{max}\to Y_\mathrm{max}$ is a closed embedding, its
image is an analytic subset of $Y_\mathrm{max}$ biholomorphic to
$X_\mathrm{max}$. Since the analytically Zariski-open subset $\Omega\subset X$
is contained in $X_\mathrm{max}$, the image $\phi(\Omega)$ equals the analytic
subset $\phi(X_\mathrm{max})\cap U\cap\Omega_Y$ of $U\cap\Omega_Y$ and
$\phi|_\Omega\colon\Omega\to\phi(\Omega)$ is biholomorphic.

By construction the map $p_Y\colon\Omega_Y\to\Omega_Y/H$ is a submersion, hence
Lemma~\ref{Lem:UniversalityofGeomQuot} applies to show that $p_Y\bigl(
\phi(\Omega)\bigr)=Q$ is an analytic subset of $p_Y(U\cap\Omega_Y)$ and that
$p\colon\Omega\to Q$ is a geometric quotient for the $H$-action on $\Omega$.
Moreover, $p\colon\Omega\to Q$ is a submersion and a topological fibre bundle;
both properties are inherited from $p_Y$. Together with another application of
Lemma~\ref{Lem:UniversalityofGeomQuot} this shows the properties listed under
(2) and (3) in the Main Theorem.

The proof is completed by showing that $Q$ is a Stein space: Since $Q$ is an
analytic subset of $p_Y(U\cap\Omega_Y)$, for this it suffices to prove that
$p_Y(U\cap\Omega_Y)$ is a Stein open subset of $\Omega_Y/H$. Recall that
$U=\pi^{-1}(\{f\not=0\})=\{\pi^*f\not=0\}$ by definition. Since
$(\pi^*f)|_{\Omega_Y}$ is $H$-invariant, there is a function $\bar
f\in\mathscr{O}_{\Omega_Y/H}(\Omega_Y/H)$ such that $\pi^*f|_{\Omega_Y}=
p_Y^*\bar f$. It follows that $p_Y(U\cap\Omega_Y)=\{\bar f\not=0\}\subset
\Omega_Y/H$, and consequently $p_Y(U\cap\Omega_Y)$ is a Stein open subset of
the Stein space $\Omega_Y/H$.
\end{proof}

Thus, we have shown the existence of a  geometric
$H$-quotient with the properties listed in parts (1) -- (3) of the Main Theorem under the assumption that $X$ is $H$-irreducible. Combining this
with the observation noted in Section~\ref{Subs:H-irr}, parts (1), (2), and (3)
of the Main Theorem are proven.

\section{Pushing down meromorphic functions}\label{section:mero}

In this section we will prove that the geometric quotient $p\colon \Omega\to Q$
constructed in the previous section additionally has the properties stated in
parts (4)--(6) of the Main Theorem, thus completing its proof. Before we do
this in Section~\ref{Subsection:mero}, we give a criterion in terms of
meromorphic functions for a densely defined holomorphic map to extend to a
weakly meromorphic map.

\subsection{Meromorphic functions and weakly meromorphic maps}

The following lemma is concerned with the relation of
Definition~\ref{defi:meromorphic} to meromorphic functions. It will be used in
the proof of part (4) of the Main Theorem in the next subsection.

\begin{lemma}\label{Lem:ConstructingWeaklyMeroMaps}
Let $X$ be a complex space, $A \subset X$ a nowhere dense analytic subset, $Y
\subset U \subset \C^N$ an analytic subset of an analytically Zariski-open
subset $U$ in $\C^N$, and $z_1, \dots, z_N$ linear coordinates on $\C^N$. Let
$\phi\colon X \setminus A \to \C^N$ be a holomorphic map with $\phi(X\setminus
A) \subset Y$. Assume that $\phi^*(z_j)$ extends to a meromorphic function on
$X$ for all $j = 1, \dots, N$. Then, $\phi\colon X\setminus A \to Y$ is weakly
meromorphic.
\end{lemma}

\begin{proof}
We consider the compactification $\mathbb{O}^N = (\P_1)^N$ of $\C^N$, the
so-called Osgood space. By assumption, $\varphi_j := \phi^*(z_j)$ is a
meromorphic function on $X$ for all $j=1, \dots N$. Hence, by
\cite[Satz 33]{RemmertAbbildungen} the map $\phi\colon X\dasharrow \mathbb{O}^N$
is a meromorphic map in the sense of Remmert and therefore in particular also
weakly holomorphic. Since $\phi(X\setminus A) \subset Y$, the map $\varphi$ is still weakly meromorphic after restricting the
range to $Y$, see~\cite[Satz
3.13]{StollMeromorpheAbbildungenI}.
\end{proof}

\subsection{Completing the proof of the Main Theorem}\label{Subsection:mero}

For the reader's convenience we recall parts (4)--(6) of the
Main Theorem in the following

\begin{prop}\label{Prop:PullbackMero}
Let $H<G$ be an algebraic subgroup of a complex-reductive Lie group $G$ and let
$X$ be an $H$-irreducible Stein $G$-space. Let $p\colon\Omega\to\Omega/H$ be
the geometric quotient constructed in Proposition~\ref{Prop:ExistenceGeomQuot}.
Then,
\begin{enumerate}
\item[(4)] the quotient map $p\colon\Omega\to\Omega/H$ extends as a weakly
meromorphic map to $X$,
\item[(5)] for every $f\in\mathscr{M}_X(X)^H$ there exists a unique $\bar{f}\in
\mathscr{M}_{\Omega/H}(\Omega/H)$ such that $f|_\Omega=p^*\bar{f}$, and
\item[(6)] the $H$-invariant meromorphic functions on $X$ separate
$H$-orbits in $\Omega$.
\end{enumerate}
\end{prop}

\begin{rem}
Since $p\colon\Omega\to\Omega/H$ is an open holomorphic map, the pull-back
$p^*$ from $\mathscr{M}_{\Omega/H}(\Omega/H)$ to $\mathscr{M}_\Omega(\Omega)^H$
is well-defined and, as the proof of Proposition~\ref{Prop:PullbackMero} will
show, an isomorphism.
\end{rem}

\begin{proof}[Proof of Proposition~\ref{Prop:PullbackMero}]
In order to prove the first claim recall from
Section~\ref{subsect:geometricquotients} that by construction $\Omega/H$ is an
analytic subset of an analytically Zariski-open subset of an affine variety $Q$
and that $p_Y\colon\Omega_Y\to Q$ is the algebraic geometric quotient whose
existence is guaranteed by Rosenlicht's theorem. In particular the quotient map
$p_Y$ extends as a rational map to $Y$. Let us choose an embedding of $Q$ into
a finite dimensional complex vector space $V$ and let $(z_1,\dotsc,z_N)$ be
linear coordinates in $V$. These induce invariant rational functions
$p_Y^*z_j$, $j \in \{1,\dotsc,N\}$, on $Y$. Since $\phi^{-1}(A)$ is a nowhere
dense analytic set in $X$ for every nowhere dense algebraic set $A\subset Y$,
the pull-back $\phi^*p_Y^*z_j=p^*z_j$ is a meromorphic function on $X$ for
every $j$, see~\cite[Chapter~6.3.3]{CAS}. Thus by
Lemma~\ref{Lem:ConstructingWeaklyMeroMaps} the map $p=(p^*z_1,\dotsc,
p^*z_N)$ extends as a weakly meromorphic map to $X$.

For the second claim let $f\in\mathscr{M}_X(X)^H$ be given. By
abuse of notation we denote by $f$ also its restriction to $\Omega$ and thus
have $f\in\mathscr{M}_\Omega(\Omega)^H$. Recall that $f$ is holomorphic on
$\dom f=\Omega\setminus P_f$. Applying Lemma~\ref{Lem:UniversalityofGeomQuot}
we see that $\overline{P_f} := p(P_f)$ is a nowhere dense analytic
subset of $\Omega/H$. Since $P:=p\times\id_{\P_1}\colon\Omega\times\P_1\to
(\Omega/H)\times\P_1$ is a geometric quotient for the $H$-action on
$\Omega\times\P_1$, Lemma~\ref{Lem:UniversalityofGeomQuot} implies that
$P(\Gamma_f)=:\ol{\Gamma_f}$ is an analytic subset of $(\Omega/H)\times\P_1$.
We will prove that it is a meromorphic graph.

We summarise our setup in the following diagram.
\begin{equation*}
\begin{xymatrix}{
 \Omega\times\mbb{P}_1\ar[d]^P & \ar@{_{(}->}[l]
 \Gamma_f\ar[r]^{p_\Omega}\ar[d]^{P|_{\Gamma_f}}  & \Omega\ar^{p}[d] \\
 (\Omega/H)\times\mbb{P}_1 & \ar@{_{(}->}[l] \ol{\Gamma_{f}}
\ar[r]^>>>>{p_{\Omega/H}}  & \Omega/H.
}
\end{xymatrix}
\end{equation*}
First, we note that $p_{\Omega/H}^{-1}(\ol{P_f})=
P\bigl(p_\Omega^{-1}(P_f)\bigr)$ is a nowhere dense analytic subset of
$\ol{\Gamma_f}$. Moreover, the restriction of $f$ to $\dom (f)$ is holomorphic,
and hence there exists a uniquely defined holomorphic function $\bar{f}$ on the
open subset $p(\dom f)=(\Omega/H)\setminus\overline{P_f}$ such that
$p^*\bar{f}=f|_{\dom f}$. It follows from the construction that over $p(\dom f)$
the graph of $\bar{f}$ coincides with $\ol{\Gamma_f}$. In summary we have shown
that $\ol{\Gamma_f}$ is a meromorphic graph over $\Omega/H$. Consequently,
there exists a meromorphic function $\bar{f}\in\mathscr{M}_{\Omega/H}
(\Omega/H)$ such that $\Gamma_{\bar{f}}=\ol{\Gamma_f}$. By construction this
function fulfills $p^*\bar{f}=f \in \mathscr{M}_\Omega(\Omega)$.

Finally, for the proof of property (6) let $H\acts x$ and $H\acts y$ be two
orbits in $\Omega$. Then there exists a $1\leq j\leq N$ such that
$p^*z_j(x)\not=p^*z_j(y)$. By the same argument as above $p^*z_j$ yields an
$H$-invariant meromorphic function on $X$ which separates $H\acts x$ and
$H\acts y$.
\end{proof}

Applying the observation from Section~\ref{Subs:H-irr} in
order to remove the assumption of $H$--irreducibility finally completes the
proof of the Main Theorem.

\end{document}